\documentclass[11pt,a4paper]{amsart}

\usepackage{amssymb}
\usepackage{amsthm}
\usepackage{amsmath}
\usepackage{amstext}

\numberwithin{equation}{section}

\theoremstyle{plain}
\newtheorem{thm}{Theorem}
\newtheorem{prop}[thm]{Proposition}
\newtheorem{cor}[thm]{Corollary}
\newtheorem{lem}[thm]{Lemma}

\newtheorem{defn}[thm]{Definition}
\newtheorem*{notation}{Notation}
\newtheorem{example}[thm]{Example}

\newtheorem*{note}{Note}

\theoremstyle{remark}
\newtheorem{rem}[thm]{Remark}

\newcommand{\C}{\mathbb{C}}
\newcommand{\R}{\mathbb{R}}

\newcommand{\N}{\mathbb{N}}

\newcommand{\p}{\partial}
\DeclareMathOperator{\grad}{\p}

\def\({(\!(}
\def\){)\!)}

\renewcommand{\k}{\mathsf{k}}

\title[THE {\L}OJASIEWICZ EXPONENT]{THE {\L}OJASIEWICZ EXPONENT  FOR WEIGHTED HOMOGENEOUS POLYNOMIAL
WITH ISOLATED SINGULARITY}

\author{Ould M Abderrahmane}

\address{D\'eparement de Math\'ematiques, Universit\'e des Sciences,
de Technologie et de Médecine BP. 880, Nouakchott, Mauritanie}

\email{ymoine@univ-nkc.mr} \subjclass[2010]{Primary  14B05, 32S05.}

\newcommand{\abstracttext}{ The purpose of this paper is to give an explicit formula of the
{\L}ojasiewicz exponent of an isolated weighted homogeneous
singularity in terms of its weights.}

\begin{document}
\begin{abstract} \abstracttext \end{abstract} \maketitle

Let $f \colon (\C^n, 0) \rightarrow (\C, 0)$ be a holomorphic
function with an isolated critical point at $0$. The {\L}ojasiewicz
exponent $L(f)$ of $f$ is by definition
$$
L(f)=\inf\{\lambda> 0\; \,:\;\, \mid\text{grad} f\mid\geq
\text{const.} \mid z\mid^{\lambda}\,\text{ near zero }\},
$$
It is well
known(see \cite{LT}) that the {\L}ojasiewicz exponent can be
calculated by means of analytic paths
\begin{equation}\label{0.1}
L(f)=\sup\left\{\frac{\text{ord}(
\text{grad}f(\varphi(t)))}{\text{ord}(\varphi(t))}\; :\; 0\neq
\varphi(t)\in \C\{t\}^n,\; \varphi(0)=0\right\},
\end{equation}
where $\text{ord}(\phi):=\inf_i\{\text{ord}(\phi_i)\}$ for $\phi\in
{\C\{t\}}^n$. By definition, we put $\text{ord}(0) = +\infty$.

{\L}ojasiewicz exponents have important applications in singularity
theory, for instance, Teissier \cite{BT} showed that
$C^0$-sufficiency degree of $f$ (i.e., the minimal integer $r$
 such that $f$ is topologically equivalent to $f + g$ for all $g$ with
ord$(g) \geq r + 1$)  is equal to $[L(f)]+1$, where $[L(f)]$ denote
integral part of $L(f)$. Despite deep research of experts in
singularity theory, it is not proved yet that {\L}ojasiewicz
exponent $L(f)$ is a topological invariant of $f$ (in contrast to
the Milnor number). An interesting  mathematical problem is to give
formulas for $L(f)$ in terms of another invariants of $f$ or an
algorithm to compute it.
 In the two-dimensional case there
are many explicit formulas for $L(f)$ in various terms (see
\cite{CK1}, \cite{CK2}, \cite{TCKYCL}, \cite{L}).
 Estimations of the {\L}ojasiewicz exponent
 in the general case can be found in \cite{OMA}, \cite{TF}, \cite{BL}, \cite{AP1}.

 The aim of this paper is to compute the {\L}ojasiewicz exponent
 for the classes of weighted
homogeneous isolated singularities in terms of the weights. In
particular, we generalize a formula for $L(f)$ of Krasi\'nski,
Oleksik and  P{\l}oski  \cite{KOP} for weighted homogeneous surface
singularity. This was already announced  by Tan, Yau and Zuo
\cite{TYZ}, but thier paper seems to have some gaps in the proof of
proposition 3.4. We were motived by their papers. However, our
considerations are based on other ideas. More precisely, we use the
notion of weighted homogenous filtration introduced by  Paunescu in
\cite{LP}, the geometric characterization of $\mu$-constancy in
\cite{LS, BT} and the result of Varchenko \cite{ANV}, which
described the $\mu$-constant stratum of weighted homogeneous
singularities in terms of the mixed Hodge structures.

Moreover, we show that the {\L}ojasiewicz exponent is invariant for
all $\mu$-constant deformation of weighted homogeneous singularity,
which gives an affirmative partial answer to  Teissier's conjecture
\cite{BT}.

\begin{notation}
To simplify the notation, we will adopt the following conventions\,:
for a function $F(z, t)$ we denote by $\grad F$ the gradient of $F$
and by $\grad_z F$ the gradient of $F$ with respect to variables
$z$.

Let $\varphi,\,\, \psi \colon(\C^n, 0) \to \R$ be two function
germs. We say that $ \varphi(x)\lesssim \psi(x)$ if there exists a
positive constant $C> 0$ and an open neighborhood $U$ of the  origin
in $\C^n$ such that $\varphi(x)\leq C \; \psi(x)$, for all $x \in
U$. We write $ \varphi(x) \sim \psi(x)$  if $\varphi(x) \lesssim
\psi(x)$ and $\psi(x)\lesssim \varphi(x)$. Finally,
$|\varphi(x)|\ll|\psi(x)|$ (when $x$ tends to $x_0$) means
$\lim_{x\to x_0}\frac{\varphi(x)}{\psi(x)}=0$.

\end{notation}

\bigskip
\section{Weighted  homogeneous filtration }
\bigskip
Let $\N$ be the set of nonnegative integers and  $\mathcal{O}_n$
denote the ring of analytic function germs $f \colon (\C^n, 0) \to
(\C, 0)$. The Milnor number of a germ $f$, denoted by $\mu(f)$, is
algebraically defined as the $\text{dim}\,\mathcal{O}_n/{J(f)}$,
where  $J(f) $ is the Jacobian ideal in $\mathcal{O}_n$ generated by
the partial derivatives $\{\frac{\grad f}{\grad
z_1},\cdots,\frac{\grad f}{\grad z_n}\}$. Let $F \colon (\C^n \times
\C, 0) \to (\C, 0)$ be the deformation of $f$ given by $F(z, t) =
f(z) + \sum c_{\nu}(t)z^{\nu}$, where $c_{\nu}\colon (\C, 0) \to
(\C, 0)$ are germs of holomorphic functions. We use the notation
$F_t(z) = F(z, t)$ when $t$ is fixed.

From now, we shall fix a system of positive integers $w=(w_{1},
\dots , w_{n})\in (\N-\{0\})^n$, the weights of variables $z_{i},$
$w(z_{i})=w_{i},$  $1\leq i\leq n$, and a positive integer $d\geq 2
w_i$ for $i=1,\dots, n$,
 then a polynomial $f \in\C[z_1,\dots, z_n]$ is called weighted
homogeneous of degree $d$ with respect the weight
$w=(w_1,\dots,w_n)$ (or type $(d; w)$) if $f$ may be written as a
sum of monomials $z_1^{\alpha_1}\cdots z_n^{\alpha_n}$ with
\begin{equation}\label{weight}
\alpha_1 w_1 +\dots +\alpha_n w_n =d.
\end{equation}

Comparing these weights with the  $w'=(w_1'\dots,w_n') $ defined in
\cite{KOP, TYZ}, from (\ref{weight}), we get $w'(z_i)=\frac{d}{w_i}$
for $i=1,\dots ,n$, so it follows that $w_i'\geq 2$ if and only if
$d\geq 2w_i$. Also, we have
$$
\max_{i=1}^n(w_i'- 1)=\max_{i=1}^n(\frac{d}{w_i}-1).
$$

We may introduce (see \cite{LP}) the function $\rho(z)=
\left(|z_1|^{\frac{2}{w_1}}+\cdots
+|z_n|^{\frac{2}{w_n}}\right)^{\frac{1}{2}} $. We also consider the
spheres associated to this $\rho$
$$
S_{r}=\{z\in \C^{n}\;:\; \,\rho(z)=r\}, \quad \; r>0.
$$

Here $\cdot$ means the weighted action, with respect to the $\C^*$
action defined below
$$
t\cdot z=(t^{w_1}z_1,\dots, t^{w_n}z_n)
$$

\begin{defn} Using $\rho$,
we define a singular Riemannian metric on $\C^n$ by the following
bilinear form
$$
\langle \rho^{w_i}\frac{\partial}{\partial x_i}\,,\,
\rho^{w_j}\frac{\partial}{\partial x_j} \rangle =
\delta_{i,j}:=\begin{cases}
1 & \text{ if } i=j\\
0 & \text{ if } i\neq j
\end{cases}
$$

 We will denote by $\text{grad}_w$ and  $\parallel\; \parallel_w$, the corresponding gradient and norm
associated with this Riemannian metric  (for more details about
these see \cite{LP}).
\end{defn}

Let $f\in \mathcal{O}_n$. We denote the Taylor expansion of $f$ at
the origin by $\sum c_{\nu}z^{\nu}$. Setting $H_j(z)=\sum
c_{\nu}z^{\nu}$ where the sum is taken over $n$ with $<w, \nu>=
w_1+\cdots+w_n=j$, we can write the weighted Taylor expansion $f$
$$
f(z)=H_d(z)+H_{d+1}(z)+\cdots\; ; H_d \neq  0.
$$
We call $d$ the weighted degree of $f$ and  $H_d$ the weighted
initial form of $f$ about the weight. Furthermore, for any $f\in
\mathcal{O}_n$ we get
\begin{equation}\label{1.1}
\|\text{grad}_wf(z)\|_w\lesssim \rho^{d_w(f)}(z),
\end{equation}
where $d_w(f)$ denotes the degree of $f$ with respect to $w$.
Indeed, as all nonzero $z$, we find $\frac{1}{\rho(z)}\cdot z\in
S_1$, moreover, we have $\frac{\partial H_j}{\partial z_i}$  is zero
or a weighted homogeneous polynomial of degree $d-w_j$, then we
obtain
$$
\|\text{grad}_w H_j(\frac{1}{\rho(z)}\cdot z)\|_w=
\frac{\|\text{grad}_w H_j(z)\|_w}{\rho(z)^j} \lesssim 1.
$$
Therefor,
$$
\parallel\text{grad}_wf(z)\parallel_w\lesssim\sum_{j\geq d_w(f)}\parallel \text{grad}_wH_j(z) \parallel_w\lesssim
\rho^{d_w(f)}(z).
$$

\begin{prop}\label{Hassane}
Let   $f\in \mathcal{O}_n$ be a weighted homogeneous isolated
singularity of type $(d; w)$ at $0\in \C^n $. Then
\begin{equation}\label{1.2}
\|\text{grad}_wf(z)\|_w \gtrsim \rho(z)^d.
\end{equation}
\end{prop}

\begin{proof}
Since $f$ has only  isolated singularity at the origin, then  for
small values of $r$  we have
\begin{equation}\label{1.3}
\| \text{grad}_w f(z)\|_w = \left( \sum_{i=1}^n |\rho^{w_i}(z)
\frac{\partial f}{\partial z_i}(z) |^2 \right)^{\frac{1}{2}}\gtrsim
1, \; \forall z\in S_r.
\end{equation}
 On the other hand, $\frac{\partial f}{\partial z_i}$ is weighted
 homogeneous of degree $d-w_i$ for $i=1,\dots, n$ and also, $\frac{r}{\rho(z)}\cdot z\in S_r$ for all nonzero $z$.
  Thus, by (\ref{1.3}) we obtain
$$
\| \text{grad}_w f(\frac{r}{\rho(z)}\cdot z)\|_w=r^d\frac{\|
\text{grad}_w f(z)\|_w}{\rho(z)^d}\gtrsim 1.
$$
This completes the proof of the proposition.
\end{proof}

\bigskip
\section{The results}
\bigskip

The main result of this paper is the following:

\begin{thm}\label{main1}
 Let $f \colon (\C^n, 0) \rightarrow (\C, 0)$
 be a weighted homogeneous polynomial of type $(d; w)$ with $d\geq 2
 w_i$ for $i=1,\dots, n$, defining an isolated singularity at the origin $0\in\C^n$. Then
$$
 L(f)=\max_{ i=1}^n (\frac{d}{w_i} - 1).
$$
\end{thm}

\begin{cor}\label{main2}
 Let $f \colon (\C^n, 0) \rightarrow (\C, 0)$
 be a weighted homogeneous polynomial of type $(d; w)$
  with $d\geq 2 w_i$ for $i=1,\dots, n$, defining an isolated singularity at the origin
in $\C^n$. For any deformation
 $F_t(z)=f(z) + \sum c_{\nu}(t)z^{\nu}$ for which $\mu(F_t)=\mu(f)$
  is called $\mu$-constant, then
 $L(F_t)$ is also constant.
\end{cor}

\begin{cor}\label{main3}
Let $f \colon (\C^n, 0) \rightarrow (\C, 0)$ be a holomorphic
function with $d_w(f)\geq 2w_i$ for $i=1,\dots, n$. If   the
weighted initial forms of $f$ define an isolated singularity at the
origin, then
$$
 L(f)=\max_{ i=1}^n (\frac{d}{w_i} - 1).
$$
\end{cor}
\bigskip
\section{Proofs of the Theorem \ref{main1},  Corollary
\ref{main2} and Corollary \ref{main3}}
\bigskip

Before starting the proofs, we will recall some important results on
 the geometric characterization of
$\mu$-constancy.

\begin{thm}[Greuel \cite{greuel}, L\^e-Saito \cite{LS}, Teissier \cite{BT}]\label{le-saito-tessier}
Let $F \colon (\C^n\times \C^m, 0) \to (\C, 0)$ be the deformation
of a holomorphic $f \colon (\C^n, 0) \to (\C, 0)$ with isolated
singularity. The following statements are equivalent.
\begin{enumerate}
\item $F$ is a $\mu$-constant deformation of $f$.

\item $\frac{\grad F}{\grad t_j}\in \overline{J(F_t)}$, where $\overline{J(F_t)}$
denotes the integral closure of the Jacobian ideal of $F_t$
generated by the partial derivatives of $F$ with respect to the
variables $z_1,\dots, z_n$.

\item  The deformation $ F(z, t) = F_t(z)$ is a Thom map, that is,
$$
\sum_{ j=1}^m |\frac{\grad F}{\grad t_j} | \ll \| \grad F\| \text{
as } (z, t) \to (0, 0).
$$

\item  The polar curve of $F$ with respect to $\{t = 0\}$ does not split,
that is,
$$
\{(z,t) \in \C^n \times \C^m \,\,| \,\,\grad_zF (z,t) = 0\} = \{0\}
\times \C^m \text{ near } (0,0).
$$

\end{enumerate}
\end{thm}

\subsection{Proof of Theorem \ref{main1}}

First, by the proposition \ref{Hassane} we get
$$
\|\text{grad}_wf(z)\|_w^2=\sum_{i=1}^n |\rho^{w_i}(z) \frac{\partial
f}{\partial z_i}(z) |^2\gtrsim \rho(z)^{2d}.
$$
Therefore,
$$
\rho(z)^{\min\{w_i\}}\| \text{grad} f(z)\|\gtrsim \rho(z)^{d}.
$$
Hence
$$
\| \text{grad} f(z)\|\gtrsim
\left(\sum_{i=1}^n|z_i|^{\frac{1}{w_i}}\right)^{d-\min\{w_i\}}
\gtrsim
 |z|^{
\frac{d-\min{w_i}}{\min{w_i}} }=|z|^{\max_{ i=1}^n (\frac{d}{w_i} -
1)},
$$
it follows that $L(f)\leq\max_{i=1}^n(\frac{d}{w_i} - 1)$.\\

In order to   show the opposite inequality we need the following
lemma.

\begin{lem}\label{Key0}
Let $f\in\mathcal{O}_n$ be a weighted homogeneous isolated
singularity of type $(d; w)$. Suppose that $w_k=\min_{i=1}^n{w_i}$
and $V_{z_k}(f)\nsubseteq\{z_k=0\}$, where
$$
V_{z_k}(f)=\left\{z\in\C^n \; :\,\frac{\partial f}{\partial
z_1}(z)=\cdots=\frac{\partial f}{\partial z_{k-1}}(z)=\frac{\partial
f}{\partial z_{k+1}}(z)=\cdots=\frac{\partial f}{\partial
z_n}(z)=0\right\}.
$$
Then $L(f)=\frac{d}{w_k} - 1=\max_{ i=1}^n (\frac{d}{w_i} - 1).$

\end{lem}

\begin{proof}
See \cite{KOP}, Proposition $2$.
\end{proof}

We now want to prove the opposite inequality.
 Modulo a permutation coordinate of $\C^n$, we may assume that $w_1\leq
w_2\leq\cdots\leq w_n$. Since $f$ be a weighted homogeneous of
degree $d$ with  isolated singularity, It is easy to check that the
monomial $z_1^{q_1}$ or $z_1^{q_1}z_i$ appear in the expansion of
$f$. There are three cases to be considered.

{\bf Case 1.} In this case, we suppose $z_1z_i$ appear in the
expansion of $f$, since $f$ defining an isolated singularity at the
origin $0\in \C^n$, there exist the terms $z_n^{q_n}$ or
$z_n^{q_n}z_j$ with non-zero coefficients in $f$.

 We first consider  the case whereby
$z_n^{q_n}z_j$ appear in $f$, from the hypotheses $d=w_1+w_i\geq
2w_n\geq\cdots \geq2w_1$, then we may write
$$
d=q_n w_n+w_j\geq( q_n-1)w_n + w_i + w_j\geq w_i +w_j\geq
w_i+w_1=d\geq w_i +w_i\geq w_1+w_i=d,
$$
Therefore, $q_n=1$ and $w_1=w_2=\cdots=w_n$.

 We will next consider  the case whereby
$z_n^{q_n}$ appear in $f$, since $\grad f(0)=0$, we have  $q_n\geq
2$, it follows that
$$
d=q_nw_n\geq (q_n-1)w_n+w_i\geq w_1+w_i=d,
$$
hence $q_n=2$ and $w_1=w_2=\cdots=w_n$.

In the homogenous case $w_1=w_2=\cdots=w_n$, for any nonzero $a\in
\C^n$, along the curve $\varphi(t)=t\cdot a=(t^{w_1}a_1,\dots,
t^{w_n}a_n)$, we obtain  $ \grad f(\varphi(t))=t^{d-w_1}\grad f(a)
$, it follows from (\ref{0.1}) that
$$
L(f)\geq \frac{\text{ord}(\grad
f(\varphi(t))}{\text{ord}(\varphi(t))} =\frac{d}{w_1}-1=\max_{
i=1}^n (\frac{d}{w_i} - 1).
$$
This ends the proof of Theorem \ref{main1} in the first case.\\

{\bf Case 2.} In this case, we suppose $z_1^{q_1}$ appear in the
expansion of $f$ and $z_1z_i$ doesn't appear for $i=2,\dots,n$. Take
an analytic path $\varphi (t)=(t,0,\dots,0)$, then from (\ref{0.1})
we get
$$
L(f)\geq \frac{\text{ord}(\grad
f(\varphi(t))}{\text{ord}(\varphi(t))} =q_1-1=\frac{d}{w_1}-1=\max_{
i=1}^n (\frac{d}{w_i} - 1).
$$
This ends the proof of Theorem \ref{main1} in the second case.\\

{\bf Case 3.} In this case,  we suppose that $z_1^{q_1}z_i$ appear
in the expansion of $f$ with $q_1\geq 2$. By lemma \ref{Key0} it is
enough to prove that $V_{z_1}(f)\nsubseteq\{z_1=0\}$. Indeed,
suppose that $V_{z_1}(f)\subset\{z_1=0\}$. Then, we let the
deformation $ F(z,t)=f(z)+tz_1^{q_1} $ of $f$. Since,
$$
V_t(F)= \{(z,t)\in\C^n\times \C \,\,| \,\,\grad_zF (z,t) =
0\}\subset V_{z_1}(F_t)=V_{z_1}(f)\subset\{z_1=0\},
$$
this means that
$$
\grad_zF(z,t)=0 \text{ if and only if } \grad f(z)= 0.
$$
Since $f$ defining an isolated singularity, and hence, by $(4)$ in
theorem \ref{le-saito-tessier} we get that $F_t$ is $\mu$-constant.
According to the result of Varchenko's theorem \cite{ANV}, the
monomial $z_1^{q_1}$ verifies $d_w(z_1^{q_1})=q_1w_1\geq  d$. But
$d=q_1w_1 +w_i>q_1w_1\geq d$, which is a contradiction. This
completes the proof of Theorem \ref{main1}.


\subsection{Proof of Corollary  \ref{main2}}

Let $f_t(z) =f(z) + \sum_{\nu}c_{\nu}(t)z^{\nu}$ be a deformation
$\mu$-constant of  a weighted homogeneous polynomial $f$ of degree
$d$ with isolated singularity. Since $c_{\nu}(0)=0$, we can write
$$
f_t(z) =f(z) + tg_t(z).
$$
By a result of Varchenko's theorem \cite{ANV}, the deformation $g_t$
verifies $d_w(g_t)\geq  d$  for all $t$. This together with
(\ref{1.1}) and (\ref{1.2}) gives
$$
\aligned \|\text{grad}_wf_t(z)\|_w&\geq\|\text{grad}_wf(z)\|_w-|
t|\|\text{grad}_wg_t(z)\|_w\\
 &\gtrsim \rho^d(z),\quad \text{ as } |t|\ll 1.
 \endaligned
$$
Moreover, by a similar argument to the proof of the first inequality
in theorem \ref{main1} we find the following :
$$
L(f_t)\leq\max_{ i=1}^n (\frac{d}{w_i} - 1).
$$
By the semicontinuity of the {\L}ojasiewicz exponent in holomorphic
$\mu$-constant families of isolated singularities \cite{AP, BT}, we
find that $L(f)=\max_{ i=1}^n (\frac{d}{w_i} - 1)\leq L(f_t)$.
 Then the result follows.
\subsection{Proof of corollary \ref{main3}}
Let $d=d_w(f)$, it says that $f$ can be writen in the form
$$
f(z)=H_d(z)+H_{d+1}(z)+\cdots; H_d\neq 0.
$$
Since
$$
\|\text{grad}_wf(z)\|_w\geq\|\text{grad}_wH_d(z)\|_w-\sum_{j>d}\|\text{grad}_w
H_j(z)\|_w.
$$
It follows from the isolated singularity of $H_d$  with (\ref{1.1})
and (\ref{1.2}), we obtain
$$
\|\text{grad}_wf(z)\|_w\gtrsim \rho^d(z).
$$
By a similar argument to the proof of the first inequality in
theorem \ref{main1} we find the following :
$$
L(f)\leq\max_{ i=1}^n (\frac{d}{w_i} - 1).
$$
From the theorem \ref{main1} and proposition $4.1$ in \cite{BKO}, we
find
$$
\max_{ i=1}^n (\frac{d}{w_i} - 1)=L(H_d)\leq L(f).
$$
This complete the proof of the corollary.

 \begin{rem}
There is another (weaker) definition of a weighted homogeneous
polynomial. A polynomial $f  \in\C[z_1,\dots,z_n]$ is called a weak
weighted homogeneous polynomial,  if there exist $n$ integers
positive (weights) $w=(w_1,\dots,w_n)$ such that
 $f$ may be written as a sum of monomials $z_1^{\alpha_1}\cdots
z_n^{\alpha_n}$ with
$$
\alpha_1 w_1 +\dots +\alpha_n w_n =d.
$$
For
three variables  $n=3$, Krasi\'nski, Oleksik and  P{\l}oski
  proof in \cite{KOP}
  that
  $$
  L(f)=\min\left(\max_{ i=1}^n (\frac{d}{w_i}
- 1),\; \mu (f)\right).
$$
 But this is not valid  for $n$ grater than
$3$, indeed,  let
$$
f(z_1,z_2,z_3,z_4)=z_1z_4 +z_1^{10}+z_2^5+z_3^5,
$$
is weak weighted homogeneous of type $(10; 1, 2, 2, 9)$.
 Since
$\mu(f)=\prod_i^n (\frac{d}{w_i}-1)$ by the Milnor-Orlik formula
\cite{MO}, then $\mu(f)=16$. Moreover, it easy to cheek that
$L(f)=4$ and  $\max_{ i=1}^n (\frac{d}{w_i} - 1)=10$, hence
$L(f)<\min(\max_{ i=1}^n (\frac{d}{w_i} - 1),\; \mu (f))$.
\end{rem}

\section{The maximal and minimal coordinates}

 The class of weak weighted homogeneous polynomials is
broader than the class of weighted homogeneous polynomials. In order
to extend our main result to this class, we introduce the maximal
and the minimal coordinate.

\begin{defn}
 Let $f  \in\C[z_1,\dots,z_n]$ be a weak weighted
homogenous of type $(d;w_1,\dots,w_n)$, we set
$$
\aligned
 M(w)&=\left\{i\in\{1,\dots,n\} \; \mid \; d<2w_i\right\}, \\
I_{max_1}&=\max\left\{i\in M(w)\right\},\; I_{max_2}=\max\left\{i\in
M(w)-\{I_{max_1}\}\right\},\dots\\
I_{max_k}&=\max\left\{i\in
M(w)-\{I_{max_1},\dots,I_{max_{k-1}}\}\right\},
\endaligned
$$
where $k$ is the cardinal of $M(w)$. We have $M(w)=\{I_{max_1},
\dots,I_{max_k}\}$. We set
$$
\aligned
 I_{min_1}&=\begin{cases} I_{max_1} \text{ if }z_iz_{I_{max_1}} \text{ don't appear in } f \;\forall i=1,\dots n,\\
  \min\left\{i\in\{1,\dots,n\} \; \mid \; z_iz_{I_{max_1}}\text{ appear in } f
  \right\}.
 \end{cases}, \\
I_{min_2} &=\begin{cases} I_{max_2} \text{ if }z_iz_{I_{max_2}} \text{ don't appear in } f \;\forall i=1,\dots n,\\
  \min\left\{i\in\{1,\dots,n\}-\{I_{min_1}\} \; \mid \; z_iz_{I_{max_2}}\text{ appear in } f
  \right\}.
 \end{cases} ,\\
\dots&\dots\\
I_{min_k} &=\begin{cases} I_{max_k} \text{ if }z_iz_{I_{max_k}} \text{ don't appear in } f \;\forall i=1,\dots n,\\
  \min\left\{i\in\{1,\dots,n\}-\{I_{min_1},\dots,I_{min_{k-1}}\} \; \mid \; z_iz_{I_{max_2}}\text{ appear in } f
  \right\}.
 \end{cases}
 .
\endaligned
$$
\end{defn}

We put $I(f)=\left\{ I_{min_1},\dots, I_{min_k}\right\}$, we define
the maximal coordinates of the variables, the $z_i$, for $i\in
M(w)$, i.e., the coordinates of weights  $w(z_i)=w_i>\frac{d}{2}$,
also we called the minimal coordinates of the variables, the $z_i$,
for $i\in I(f)$. Finally, we set $M(f)=M(w)\cup I(f)$, $\ell(f)$ the
cardinal of $M(f)$ and
$w_{M(f)}=(w_1,\dots,\widehat{w_k},\dots,w_n)$, where the hat means
omission of all $w_k$ such that $\k\in M(f)$.

Now we are ready to extend our main result.

\begin{thm}\label{main4}
Let $f\in \mathcal{O}_n$ be a weak weighted homogeneous polynomial
of type $(d; w_1,\dots,w_n)$ defining an isolated singularity at the
origin. Then
$$
L(f)=\begin{cases}
\max_{ i\notin M(f)} (\frac{d}{w_i} - 1)& \text{ if } \ell(f)<n\\
1 & \text{ if } \ell(f)=n.
\end{cases}
$$
\end{thm}
Note that if $d \geq2w_i$ for all $i =1,\dots,n$, then
$M(f)=\emptyset$ and we recover theorem \ref{main1}.

\begin{proof}
Without loss of generality, we suppose that $w_1\leq
w_2\leq\cdots\leq w_n$. By the proof of the main result, only the
first case in the opposite inequality can be considered. This
remains us to consider the case where $z_1z_i$ appears in the
expansion of $f$.

 Let $d<2w_{n-k+1}\leq\cdots\leq 2w_n$, and so $M(w)=\{n, \dots,n-k+1\}=\{I_{max_1},
\dots,I_{max_k}\}$. Since $f$ defining an isolated singularity, it
is easy to check that the monomial $z_n^q$ or $z_n^qz_j$ appear in
expansion of $f$, then $qw_n=d$ or $qw_n+w_j=d$, but $\grad f (0)=0$
and $d<2w_n$, so that $z_nz_j$ appear in $f$. Moreover, for any
monomial $z_1^{\alpha_1}\cdots z_n^{\alpha_n}$ of $f$ with
${\alpha_n}\neq 0$, we have
$$
2w_n>d=\sum_{j<n}\alpha_jw_j +\alpha_nw_n\geq
\left(\sum_{j<n}\alpha_j\right)w_1+ w_n\geq w_1 +w_i=d.
$$
Then, $\sum_{j<n}\alpha_j=1$, ${\alpha_n}=1$ and $w_i=w_n$.

Therefore we may write
$$
f(z)=a_{I_{min_1}}z_{I_{min_1}}z_n + \sum_{j\neq I_{min_1}} a_j
z_jz_n+ f(z_1,\dots, z_{n-1}, 0),\quad a_{I_{min_1}}\neq0,
$$
for $a_j\neq0$, we have $d=w_n+w_{I_{min_1}} =w_n+w_j=w_1+w_i$, so
we obtain $w_j=w_{I_{min_1}}=w_1$. After permutation of coordinates
with same weights it can be written as
$$
f(z)=z_1z_n + \sum_{j>1} a_j z_jz_n+ f(z_1,\dots, z_{n-1}, 0),
$$
 Then we may assume, by a change of coordinates
$\xi_{1}=z_{1} + \sum_{j>1} a_jz_j$, that $ f(z)=z_1z_n+ f(z_1,\dots
z_{n-1}, 0)=z_1(z_n + g(z))+f(0,z_2\dots,z_{n-1}, 0)$
 also by a change of
coordinates $\xi_n=z_n +g(z)$, we can assume that
$$
f(z)=z_1z_n +f(0,z_2,\dots, z_{n-1},0).
$$
We set $h(z_1,\dots,z_{n-2})=f(0,z_1,\dots,z_{n-2},0)$, obviously
implies $L(f)=L(h)$. For $M(f)\neq\{1,\dots,n\}$, it follows by
elimination of the maximal and minimal coordinates that $L(f)=L(h)$,
where
 $h\in\mathcal{O}_{n-\ell(f)}$ be weighted homogenous of type
 $(d;w_{M(f)})$. Therefore by theorem \ref{main1}, we get
$$
 L(f)=L(h)=\max_{ i\notin M(f)} (\frac{d}{w_i} -
 1).
 $$
 For $M(f)=\{1,\dots,n\}$, then we can suppose, by   the splitting lemma,
 that $f(z)=z_1^2+\dots+z_n^2$, thus $L(f)=1$. The Theorem \ref{main4} is
 proved.
\end{proof}

\begin{example}
Let
$$
f(z)=z_1z_6+z_1^{12}+z_2z_5+z_3^4+z_4^3+z_2^6,
$$
$f$ is weak weighted homogenous of type $(12;1,2,3,4,10,11)$ with
isolated singularity, since $M(w)=\{5,6\}$ and
$M(f)=\{1,2,5,6\}\subsetneq\{1,\dots,6\}$, then by theorem
\ref{main4} we get
 $$
 L(f)=\max_{ i\notin M(f)} (\frac{d}{w_i}-1) =3.
 $$
\end{example}
\begin{example}
Let
$$
f(z)=z_1z_6+z_2z_5+z_3z_4,
$$
$f$ is weak weighted homogenous of type $(12;1,2,3,9,10,11)$
defining an isolated singularity, since $M(w)=\{4,5,6\}$ and
$M(f)=\{1,\dots,6\}$, then by theorem \ref{main4} we get $L(f)=1$.
Also $f$ can be seen as weighted homogenous of type
$(2;1,1,1,1,1,1)$, and hence by theorem \ref{main1}, $L(f)=1$.
\end{example}

\begin{rem}

For a weak weighted homogeneous polynomials $f \in\C[z_1,\dots,z_n]$
defining an isolated singularity, our proof of the first equality in
the main theorem is valid. Since it is well known that $\mu(f)\geq
L(f)$, it follows that
\begin{equation}\label{n=3}
L(f)\leq\min\left(\max_{ i=1}^n (\frac{d}{w_i} - 1), \mu(f) \right),
\end{equation}
If $\ell(f)=0$, by theorem \ref{main4}, we obtain $L(f)=\max_{
i=1}^n (\frac{d}{w_i} - 1)$. Also, if $\ell(f)=n-1$, by splitting
lemma, we have $L(f)=\mu(f)$.

 Moerover, it easy to check that $\ell(f)$ is even, then for $n=3$, we have
$\ell(f)=0$ or $2$. Then, we get the result of Krasi\'nski, Oleksik
and  P{\l}oski
   in \cite{KOP}, that is
$$
L(f)=\min\left(\max_{ i=1}^n (\frac{d}{w_i} - 1), \mu(f) \right).
$$

Finally, using the processus of elimination  of the maximal and
minimal coordinates, the result of corollary \ref{main2} and
\ref{main3} can be extended  to the  class of weak weighted
homogeneous polynomials i.e., we drop the hypothesis $d\geq 2w_i$
for $i=1,\dots,n$ in the corollaries.

\end{rem}

\begin{note}
 A. Parusi\'nski called my attention to S. Brzostowski's
result \cite{SB}, which has independently proved  the main  theorem
of  this  paper. But his proof is different from ours.

\end{note}


\bigskip
\begin{thebibliography}{99}
\bibitem{OMA} O.M. Abderrahmane, \emph{ On the {\L}ojasiewicz exponent and Newton
polyhedron,} Kodai Math. J., \textbf{28} (2005), 106--110.

\bibitem{SB} Szymon Brzostowski, \emph{The {\L}ojasiewicz Exponent of
Semiquasihomogeneous Singularities,}
http://arxiv.org/abs/1405.5179v1, May 2014.

\bibitem{BKO} Szymon Brzostowski, Tadeusz Krasi\'nski and Grzegorz
Oleksik, \emph{ A conjecture on the Lojasiewicz exponent,} Journal
of Singularities \textbf{6} (2012), 124--130.


\bibitem{CK1} J. Chadzy\'nski and T. Krasi\'nski, \emph{The {\L}ojasiewicz
exponent of an analytic mapping of two complex variables at an
isolated zero,} In: Singularities, Banach Center Publ. \textbf{20},
PWN, Warszawa, 1988, 139–-146.

\bibitem{CK2} J. Chadzy\'nski and T. Krasi\'nski, \emph{Resultant and the {\L}ojasiewicz exponent,} Ann. Polon. Math.
\textbf{61} (1995), 95–100.

\bibitem{TF}  T. Fukui, \emph{{\L}ojasiewicz type inequalities and Newton diagrams,}
Proc. Amer. Math. Soc., \textbf{112} (1991), 1169--1183.

\bibitem{greuel} G-M. Greuel, \emph{Constant Milnor Number Implies Constant
Multiplicity For Quasihomogeneous Singularities,} Manuscritpta Math.
\textbf{56} (1986), 159--166.


\bibitem{KOP} T.Krasi\'nski, G. Oleksik and A. P{\l}oski, \emph{The {\L}ojasiewicz exponent of
an isolated weighted homogeneous surface singularity, } Proc. Amer.
Math. Soc. \textbf{137} (2009), 3387--3397. [KL] Kuo, T. C. and Lu,

\bibitem{TCKYCL} T. C. Kuo and Y. C. Lu, \emph{On analytic function germs of
two complex variables,} Topology \textbf{16} (1977), 299–-310.

\bibitem{LT} M. Lejeune-Jalabert and B. Teissier, \emph{Cl\^oture integrale des
id\'eaux et \'equisingularite, } in: S\'eminaire Lejeune-Teissier,
Centre de Math\'ematiques, \'Ecole Polytechnique, Universit\'e
Scientifique et Medicale de Grenoble, 1974.

\bibitem{LR} D. T. L\^e and C. P. Ramanujam, \emph{Invariance of Milnor's number
implies the invariance of topological type,} Amer. J. Math.
\textbf{98} (1976), 67--78.

\bibitem{LS} D.T. L\^e and K. Saito,
\emph{La constence du nombre de Milnor donne des bonnes
stratifications,} Compt. Rendus Acad. Sci. Paris, s\'erie A
\textbf{272} (1973), 793--795.

 \bibitem{L} A. Lenarcik, On the \emph{{\L}ojasiewicz exponent of the gradient of a holomorphic function,} In:
Singularities Symposium–{\L}ojasiewicz 70. Banach Center Publ.
\textbf{44}, PWN, Warszawa, 1998, 149–-166.


\bibitem{BL} B. Lichtin, \emph{Estimation of ?Lojasiewicz exponents and Newton
polygons,} Invent. Math., \textbf{64} (1981), 417-429.


\bibitem{JM} J. Milnor, \emph{Singular points of complex hypersurfaces}, Ann. of Math.
Stud. \textbf{61} (1968), Princeton Univer- sity Press. [MO] Milnor,

\bibitem{MO}J. Milnor and P. Orlik,  \emph{Isolated singularities de?ned by
weighted homogeneous polynomials,} Topology \textbf{9} (1970),
385–-393.

\bibitem{DBO} D. B. O'Shea, \emph{Topologically Trivial Deformations
of Isolated Quasihomogeneous Singularities Are Equimultiple,} Proc.
A.M.S. \textbf{101}:2 (1987), 260--262.

\bibitem{LP} L. Paunescu,
\emph{A weighted version of the Kuiper-Kuo-Bochnak-{\L}ojasiewicz
theorem,} J. Algebr. Geom. \textbf{2}, 69--79 (1993)

\bibitem{AP1} A. P{\l}oski, \emph{Sur l'exposant d'une application analytique.
I,} Bull. Polish Acad. Sci. Math., \textbf{32} (1984), 669--673.

\bibitem{AP} A. Ploski, \emph{Semicontinuity of the Lojasiewicz
exponent,} Univ. Iagel. Acta Math. \textbf{48} (2010), 103--110.

\bibitem{TYZ} S. Tan, S. S.-T. Yau and H. Zuo, \emph{{\L}ojasiewicz inequality for weighted
homogeneous polynomial with isolated singularity,} Proc. Amer. Math.
Soc. \textbf{138} (2010), 3975--3984.


\bibitem{BT} B. Teissier, \emph{Vari\'et\'es polaires, }  Invent. Math. \textbf{40} (1977), 267–292


\bibitem{ANV} A. N. Varchenko, \emph{A lower bound for the codimension of the stratum
$\mu$-constant in term of the mixed Hodge structure,} Vest. Mosk.
Univ. Mat. \textbf{37} (1982), 29--31



\end{thebibliography}
\end{document}